\documentclass{article}

\usepackage{amsfonts}
\usepackage{amsthm}
\usepackage{amsmath}

\theoremstyle{theorem}
\newtheorem{theorem}{Theorem}

\newtheorem{lemma}{Lemma}

\theoremstyle{definition}

\newtheorem*{remark}{Remark}
\newtheorem{assumption}{Assumption}

\begin{document}

\title{A simple proof of existence of Lagrange multipliers}
\markright{A simple proof of existence of Lagrange multipliers}
\author{Gabriel Haeser${}^1$ and Daiana Oliveira dos Santos${}^2$}

\date{August 16, 2022\\[0.2cm]
 ${}^1$ Department of Applied Mathematics, University of S\~ao Paulo, S\~ao Paulo-SP,
Brazil. Email: ghaeser@ime.usp.br\\
 ${}^2$ Paulista School of Politics, Economics and Business, Federal University of S\~ao
Paulo, Osasco-SP, Brazil. Email: daiana.santos@unifesp.br
}

\maketitle

In the seminal book {\it M\'echanique analitique, Lagrange, 1788}, the notion of a Lagrange multiplier was first introduced in order to study a smooth minimization problem subject to equality constraints. The idea is that, under some regularity assumption, at a solution of the problem, one may associate a new variable (Lagrange multiplier) to each constraint such that an equilibrium equation is satisfied. This concept turned out to be central in studying more general constrained optimization problems and it has lead to the rapid development of nonlinear programming as a field of mathematics since the works of Karush \cite{karush} and Kuhn-Tucker \cite{kuhn-tucker}, who considered equality and inequality constraints. The usual proofs for the existence of Lagrange multipliers are somewhat cumbersome, relying on the implicit function theorem or duality theory. In the first section of this note we present an elementary proof of existence of Lagrange multipliers in the simplest context, which is easily accessible to a wide variety of readers. In addition, this proof is readily extended to the much more general context of conic constraints, which we present in the second section together with the background properties needed on the projection onto a closed and convex cone.

\section{Lagrange multipliers for equality constraints}
Let us start by considering the problem
\begin{equation}\label{eq}\begin{aligned}
\mathop{\mbox{Minimize }}_{x\in\mathbb{R}^n}&f(x),\\
\mbox{subject to }&h_1(x)=0,\dots,h_m(x)=0,
\end{aligned}\end{equation}
where $f,h_1,\dots,h_m:\mathbb{R}^n\to\mathbb{R}$ are continuously differentiable functions. Denoting $h:=(h_1,\dots,h_m)$, our goal is to show that when $\overline{x}$ is a local solution of \eqref{eq}, that is, $h(\overline{x})=0$ and $f(\overline{x})\leq f(x)$ for all $x$ sufficiently close to $\overline{x}$ such that $h(x)=0$, under some regularity condition, there exist so-called Lagrange multipliers $\lambda_1,\dots\lambda_m\in\mathbb{R}$ such that
\begin{align}
\label{lagrange}
\nabla f(\overline{x})+\sum_{i=1}^m\lambda_i\nabla h_i(\overline{x})=0.
\end{align}
That is, at $\overline{x}$, the gradient of $f(\cdot)$ is a linear combination of the gradients of $h_i(\cdot), i=1,\dots,m$, where we use $\nabla$ to denote the gradient operator. The regularity condition we will employ is the following:
\begin{assumption}
\label{licq}
The gradients of the constraints at $\overline{x}$, that is, $\nabla h_1(\overline{x}),\dots,\nabla h_m(\overline{x})$, are linearly independent.
\end{assumption}
After the proof we will discuss how this assumption can be relaxed. Our goal is to prove:
\begin{theorem}\label{lagreq}Let $\overline{x}$ be a local solution of \eqref{eq} that satisfies Assumption \ref{licq}. Then, there exist so-called Lagrange multipliers $\lambda_1,\dots,\lambda_m\in\mathbb{R}$ such that \eqref{lagrange} holds.
\end{theorem}

The proof, whose main ingredients can be found in \cite{andreani2011}, consists in building a sequence $\{x^k\}_{k\in\mathbb{N}}\to\overline{x}$ where, for each $k\in\mathbb{N}$, $x^k$ will be a local solution for the unconstrained minimization of $f(\cdot)$ plus a penalization term that increasingly forces the fulfillment of the constraints. The fact that the derivative of this function should vanish at $x^k$ will be enough for defining a sequence of approximate Lagrange multipliers which will be shown to be bounded under our assumption, yielding true Lagrange multipliers at its limit points. We will use $\|\cdot\|$ to denote the euclidean norm on any euclidean space.

\begin{proof}
Let $\delta>0$ be such that $f(\overline{x})\leq f(x)$ for all $x$ with $h(x)=0$ and $\|x-\overline{x}\|\leq\delta$, and let us consider the following sequence of penalized subproblems, for $k\in\mathbb{N}$
\begin{align*}
\mathop{\mbox{Minimize }}_{x\in\mathbb{R}^n}&\phi_k(x):=f(x)+\frac{1}{2}\|x-\overline{x}\|^2+\frac{k}{2}P(x),\\
\mbox{subject to }&\|x-\overline{x}\|\leq\delta,
\end{align*}
where $P(x):=\|h(x)\|^2$. By Weierstrass' extreme value theorem, since $\phi_k(\cdot)$ is continuous and the constraint set is compact, for each $k\in\mathbb{N}$, let $x^k$ be a global solution of the above problem and let us show that $\{x^k\}_{k\in\mathbb{N}}\to\overline{x}$.

By the boundedness of $\{x^k\}_{k\in\mathbb{N}}$, let $x^*$ be an arbitrary limit point of this sequence, say, $\{x^k\}_{k\in K_1}\to x^*$ for some infinite set of indexes $K_1\subseteq\mathbb{N}$, and let us show that $x^*=\overline{x}$. First, note that for all $k\in\mathbb{N}$, $$f(x^k)+\frac{1}{2}\|x^k-\overline{x}\|^2\leq\phi_k(x^k)\leq\phi_k(\overline{x})=f(\overline{x}).$$ Since $\{\phi_k(x^k)\}_{k\in K_1}$ is bounded from above, by continuity of the functions the only possibility is that $P(x^*)=0$. This implies that $h(x^*)=0$ and since we also have that $\|x^*-\overline{x}\|\leq\delta$, we conclude that $f(\overline{x})\leq f(x^*)$. But taking the limit for $k\in K_1$ in the above expression we arrive at $f(x^*)+\frac{1}{2}\|x^*-\overline{x}\|^2\leq f(\overline{x})\leq f(x^*)$, which implies that $x^*=\overline{x}$. This shows that the limit point $x^*$ is unique and hence the whole sequence $\{x^k\}_{k\in\mathbb{N}}$ converges to $\overline{x}$.

Now, for $k\in\mathbb{N}$ large enough we must have $\|x^k-\overline{x}\|<\delta$ and hence $x^k$ locally minimizes $\phi_k(\cdot)$ without constraints. This clearly implies, for $k\in\mathbb{N}$ large enough, that $\nabla\phi_k(x^k)=0$, which yields
\begin{equation}
\label{approx}
\nabla f(x^k)+x^k-\overline{x}+\sum_{i=1}^m\lambda_i^k\nabla h_i(x^k)=0,
\end{equation}
with $\lambda_i^k:=kh_i(x^k), i=1,\dots,m$. Let $\lambda^k:=(\lambda_1^k,\dots,\lambda_m^k)\in\mathbb{R}^m$ and let us show that $\{\lambda^k\}_{k\in\mathbb{N}}$ is bounded. If this is not the case, let us take a suitable infinite subset $K_2\subseteq\mathbb{N}$ such that $\{\|\lambda^k\|\}_{k\in K_2}\to+\infty$ and $\left\{\frac{\lambda^k}{\|\lambda^k\|}\right\}_{k\in K_2}\to\alpha=(\alpha_1,\dots,\alpha_m)\in\mathbb{R}^m$. Clearly, $\alpha\neq0$, since it is the limit of length one vectors. Thus, dividing both sides of \eqref{approx} by $\|\lambda^k\|$ and taking the limit for $k\in K_2$, by continuity of the gradients we arrive at
$\sum_{i=1}^m\alpha_i\nabla h_i(\overline{x})=0,$
which contradicts our assumption. Thus, considering a subsequence such that $\{\lambda^k\}$ converges to some $\lambda\in\mathbb{R}^m$ and taking the correspondent limit in \eqref{approx} we arrive at the result.
\end{proof}

\begin{remark}
Assumption \ref{licq} can be replaced by the weaker condition that the rank of the set $\{\nabla h_i(x)\}_{i=1}^m$ is locally constant for $x$ around $\overline{x}$, where one must simply rewrite \eqref{approx} in order for it to hold with the sum $\sum_{i=1}^m\lambda_i^k\nabla h_i(x^k)$ replaced by some $\sum_{i\in E}\tilde{\lambda}_i^k\nabla h_i(x^k)$, where $E\subseteq\{1,\dots,m\}$, $\{\nabla h_i(\overline{x})\}_{i\in E}$ is linearly independent, and $\tilde\lambda_i^k\in\mathbb{R}, i\in E$. See \cite{andreani2012}. This idea however does not extend to the more general context we pursue next.
\end{remark}

\section{Extension to conic constraints}

Now, let us consider the conic optimization problem
\begin{equation}
\label{prob}
\begin{aligned}
\mathop{\mbox{Minimize }}_{x\in X}&f(x),\\
\mbox{subject to }&h(x)\in \mathcal{K},
\end{aligned}
\end{equation}
where $f:X\to\mathbb{R}$ and $h:X\to Y$ are continuously differentiable functions, $X$ and $Y$ are real finite dimensional vector spaces equipped with corresponding inner products $\langle\cdot,\cdot\rangle$ and the associated norms $\|\cdot\|$, while $\mathcal{K}\subseteq Y$ is a closed and convex cone. In this setting, when $\overline{x}$ is a local solution of \eqref{prob}, under some regularity condition, we will show that there exists a so-called Lagrange multiplier $\lambda\in Y$ such that:
\begin{align}
\nabla f(\overline{x})+Dh(\overline{x})^*\lambda=0,\label{1}\\
\langle h(\overline{x}),\lambda\rangle=0,\label{2}\\
\lambda\in\mathcal{K}^\circ.\label{3}
\end{align}
Here, $Dh(\overline{x})$ is the derivative of $h$ at $\overline{x}$, $Dh(\overline{x})^*$ denotes its adjoint operator, and $\mathcal{K}^\circ:=\{w\in Y: \langle w,d\rangle\leq0 , \forall d\in \mathcal{K}\}$ is the polar cone of $\mathcal{K}$, which is closed and convex. The gradient $\nabla f(\overline{x})\in X$ is defined as the unique element of $X$ such that $Df(\overline{x})d=\langle \nabla f(\overline{x}),d\rangle$ for all $d\in X$. Condition \eqref{1} is sometimes called Lagrange's equation, while \eqref{2} is known as the complementarity condition and \eqref{3} as dual feasibility. Notice that when $\mathcal{K}:=\{0\}$ we have $\mathcal{K}^\circ=Y$, thus (\ref{1}--\ref{3}) recovers the previously defined notion of Lagrange multipliers for equality constraints but in a more general ambient space. When $\mathcal{K}:=\{0\}^m\times-\mathbb{R}^p_{+}\subset\mathbb{R}^{m+p}$, where $\mathbb{R}_{+}$ is the set of non-negative real numbers, we recover the standard nonlinear programming problem with $m$ equality constraints and $p$ inequality constraints. In this context, $\mathcal{K}^\circ=\mathbb{R}^m\times\mathbb{R}^p_{+}$, and we recover the Karush/Kuhn-Tucker conditions. Other important classes of conic optimization problems include, but are not limited to, optimization over the cone $\mathcal{K}$ of positive semidefinite matrices when $Y$ is the set of $m\times m$ symmetric matrices or optimization over the Lorentz cone $\mathcal{K}:=\{(x_1,\dots,x_m):x_1\geq\|(x_2,\dots,x_m)\|\}$ in $\mathbb{R}^m$.

The proof for this more general case, which is adapted from \cite{andreani2022}, is essentially the same; we simply need to adjust the penalization function $P(\cdot)$ defined in the proof to $P(x):=\|\Pi_{\mathcal{K}^\circ}(h(x))\|^2,$ which will coincide with the squared distance of $h(x)$ to $\mathcal{K}$, where $\Pi_{C}(\cdot)$ denotes the orthogonal projection onto the closed and convex set $C\subseteq Y$. The derivative of $P(\cdot)$ will be computed using the third item of the next lemma, which will allow us to define a sequence of approximate Lagrange multipliers by $\lambda^k=k\Pi_{\mathcal{K}^\circ}(h(x^k)), k\in\mathbb{N}$, implying dual feasibility \eqref{3}. The remaining new ingredients of the proof will also be shown in the next lemma, where the second item is used to justify the definition of $P(\cdot)$, in particular, characterizing the points $x\in X$ satisfying the constraint $h(x)\in\mathcal{K}$ as those such that $P(x)=0$, while the first item is responsible for the complementarity condition \eqref{2}.  The reader interested in the Karush/Kuhn-Tucker conditions may skip the proof of this lemma as all statements are straightforward when $\mathcal{K}:=\{0\}^m\times-\mathbb{R}^p_{+}$ (and $\mathcal{K}^\circ=\mathbb{R}^m\times\mathbb{R}^p_{+}$).
\begin{lemma}\label{moreau}For any $z\in Y$, the following hold: \hspace{0.3cm}i) $\langle\Pi_{\mathcal{K}}(z),\Pi_{\mathcal{K}^\circ}(z)\rangle=0$;\\
 ii) $\displaystyle\|\Pi_{\mathcal{K}^\circ}(z)\|=\inf_{u\in\mathcal{K}}{\|z-u\|}$;\hspace{2.53cm} iii) $\nabla \|\Pi_{\mathcal{K}^\circ}(z)\|^2=2\Pi_{\mathcal{K}^\circ}(z)$.
\end{lemma}
\begin{proof}
We will make use of the following well known characterization of the projection onto a closed and convex set $C\subseteq Y$: $\langle z-w,c-w\rangle\leq0$ for all $c\in C$ if, and only if $w=\Pi_{\mathcal C}(z)$.

Since $\alpha\Pi_{\mathcal{K}}(z)\in\mathcal{K}$ for all $\alpha>0$, we have $\langle z-\Pi_{\mathcal{K}}(z),\alpha\Pi_{\mathcal{K}}(z)-\Pi_{\mathcal{K}}(z)\rangle\leq0$, that is, $(\alpha-1)\langle z-\Pi_{\mathcal{K}}(z),\Pi_{\mathcal{K}}(z)\rangle\leq0$ for some positive $\alpha<1$ and some $\alpha>1$, which implies $\langle z-\Pi_{\mathcal{K}}(z),\Pi_{\mathcal{K}}(z)\rangle=0$.

Defining $w=z-\Pi_{\mathcal{K}}(z)$, it is sufficient to conclude the first two statements to show that $w=\Pi_{\mathcal{K}^{\circ}}(z)$, which follows from the fact that for all $c\in\mathcal{K}^\circ$, $\langle z-w,c-w\rangle=\langle\Pi_{\mathcal{K}}(z),c-w\rangle=\langle\Pi_{\mathcal{K}}(z),c\rangle\leq0$, where the last inequality comes from the definition of the polar cone.

To compute the derivative of $\|\Pi_{\mathcal{K}^\circ}(z)\|^2$, let $h\in Y$ and note that the definition of the projection gives $\|z-\Pi_{\mathcal{K}}(z)\|\leq\|z-\Pi_{\mathcal{K}}(z+h)\|$ and $\|z+h-\Pi_{\mathcal{K}}(z+h)\|\leq\|z+h-\Pi_{\mathcal{K}}(z)\|$. Now, a straightforward calculation shows that
\begin{align*}& 2\langle\Pi_{\mathcal{K}}(z)-\Pi_{\mathcal{K}}(z+h),h\rangle+\|h\|^2&\\
&\leq\|z+h-\Pi_{\mathcal{K}}(z+h)\|^2-\|z-\Pi_{\mathcal{K}}(z)\|^2-\langle 2(z-\Pi_{\mathcal{K}}(z)),h\rangle\leq\|h\|^2.
\end{align*}
By the Cauchy-Schwarz inequality and $1$-Lipschitzness of the projection, it follows that $\langle\Pi_{\mathcal{K}}(z)-\Pi_{\mathcal{K}}(z+h),h\rangle\geq-\|h\|^2$, which shows that $\nabla \|z-\Pi_{\mathcal{K}}(z)\|^2=2(z-\Pi_{\mathcal{K}}(z))$, and the proof is completed.
\end{proof}

The first two items in the previous lemma are due to \cite{moreau} while the third item is due to \cite{projection}. Now, the linear independence assumption in the conic context for $\overline{x}$ such that $h(\overline{x})\in\mathcal{K}$ should be given considering complementarity and dual feasibility, in the following way:
\begin{equation}\label{rob}Dh(\overline{x})^*\alpha=0,\langle h(\overline{x}),\alpha\rangle=0, \alpha\in\mathcal{K}^\circ\Rightarrow\alpha=0,\end{equation}
however, we state a more geometric but equivalent condition as follows, known as Robinson's condition \cite{rob76}:
\begin{assumption}\label{robinson}At $\overline{x}$ such that $h(\overline{x})\in\mathcal{K}$, one has $0\in\mbox{int}(Dh(\overline{x})X+\mathcal{K}-h(\overline{x})),$ where $Dh(\overline{x})X$ denotes the image space of $Dh(\overline{x})$ and $\mbox{int}(\cdot)$ denotes the interior of the underlying set.
\end{assumption}

Notice that when $\mathcal{K}:=\{0\}$, this reduces to the surjectivity of $Dh(\overline{x})$, which is equivalent to Assumption \ref{licq} in the euclidean setting. In general, Assumption \ref{robinson} means that perturbing the point $h(\overline{x})\in\mathcal{K}$, one may get back to approximately satisfying the constraints near $\overline{x}$, that is, for any sufficiently small perturbation $\varepsilon\in Y$, there exists some $d\in X$ such that $h(\overline{x})+Dh(\overline{x})d+\varepsilon\in\mathcal{K}$, where $h(\overline{x}+d)\approx h(\overline{x})+Dh(\overline{x})d$. Assumption \ref{robinson} can also be viewed as the metric regularity of the set-valued mapping $-h(\cdot)+\mathcal{K}$ at $(\overline{x},0)$, which is a well-known concept from mathematical analysis.

\begin{theorem} Let $\overline{x}$ be a local solution of problem \eqref{prob} that satisfies Assumption \ref{robinson}. Then there exists a so-called Lagrange multiplier $\lambda\in Y$ such that \eqref{1}, \eqref{2}, and \eqref{3} hold.
\end{theorem}

\begin{proof}
One can follow the proof of Theorem \ref{lagreq} by replacing the function $P(x)$ by $P(x):=\|\Pi_{\mathcal{K}^\circ}(h(x))\|^2$ in order to build a sequence $\{x^k\}_{k\in\mathbb{N}}\to\overline{x}$ such that $\nabla\phi_k(x^k)=0$ for all $k\in\mathbb{N}$ large enough, where $\phi_k(x):=f(x)+\frac{1}{2}\|x-\overline{x}\|^2+\frac{k}{2}P(x)$. Defining $\lambda^k:=k\Pi_{\mathcal{K}^\circ}(h(x^k))\in\mathcal{K}^\circ$ and computing the derivative of $P(\cdot)$ we arrive at $\nabla f(x^k)+x^k-\overline{x}+Dh(x^k)^*\lambda^k=0$ for all $k\in\mathbb{N}$ large enough. By Lemma \ref{moreau}, we have in addition that $\langle \Pi_{\mathcal{K}}(h(x^k)), \lambda^k\rangle=0$. Notice that $\{\Pi_{\mathcal{K}}(h(x^k))\}_{k\in\mathbb{N}}$ converges to $\Pi_{\mathcal{K}}(h(\overline{x}))=h(\overline{x})$ and the result follows if $\{\lambda^k\}_{k\in\mathbb{N}}$ is bounded by simply taking the limit at a convergent subsequence.

Assume by contradiction that $\{\lambda^k\}_{k\in\mathbb{N}}$ is unbounded, and let us take a suitable subsequence $K_1\subseteq\mathbb{N}$ such that $\{\|\lambda^k\|\}_{k\in K_1}\to+\infty$, $\left\{\frac{\lambda^k}{\|\lambda^k\|}\right\}_{k\in K_1}\to\alpha\in\mathcal{K}^\circ, \alpha\neq0$, and $Dh(\overline{x})^*\alpha=0$ with $\langle h(\overline{x}),\alpha\rangle=0$. The proof would end here by contradiction if we were assuming \eqref{rob}, however, by Assumption \ref{robinson}, take $t>0$ sufficiently small such that $t\alpha=Dh(\overline{x})d+w-h(\overline{x})$ for some $d\in X$ and $w\in\mathcal{K}$. Thus, $0<\langle t\alpha,\alpha\rangle=\langle Dh(\overline{x})d,\alpha\rangle+\langle w,\alpha\rangle-\langle h(\overline{x}),\alpha\rangle\leq0,$ since $\langle Dh(\overline{x})d,\alpha\rangle=\langle d,Dh(\overline{x})^*\alpha\rangle=0$, $\langle h(\overline{x}),\alpha\rangle=0$, and $\langle w,\alpha\rangle\leq0$ due to the fact that $w\in\mathcal{K}$ and $\alpha\in\mathcal{K}^\circ$. This contradiction concludes the proof.\end{proof} 

We note that it is easy to see by \eqref{lagrange} that under Assumption \ref{licq} for problem \eqref{eq}, Lagrange multipliers are unique. However, under Assumption \ref{robinson} for problem \eqref{prob} one has only compactness of the set of Lagrange multipliers (which follows similarly to our proof). Uniqueness is guaranteed when instead of the conic linear independence \eqref{rob} one requires standard linear independence, namely, requiring the implication \eqref{rob} to hold for $\alpha\in Y$ instead of for $\alpha\in \mathcal{K}^\circ$. We end by noting that this proof has inspired several new necessary optimality conditions in many other contexts; see, for instance, \cite{kanzow}, for an extension to infinite dimensional spaces.

\section*{Acknowledgment}
We would like to thank Andr\'e Salles de Carvalho for several suggestions in a first version of this notes which greatly improved our presentation.

%
\vfill\eject

\end{document}